\documentclass[11pt]{article}

\usepackage{amsmath,amsthm,amsfonts,latexsym,stmaryrd}

\begin{document}

\newtheorem*{theo}{Theorem}
\newtheorem*{pro} {Proposition}
\newtheorem*{cor} {Corollary}
\newtheorem*{lem} {Lemma}
\newtheorem{theorem}{Theorem}[section]
\newtheorem{corollary}[theorem]{Corollary}
\newtheorem{lemma}[theorem]{Lemma}
\newtheorem{proposition}[theorem]{Proposition}
\newtheorem{conjecture}[theorem]{Conjecture}

\theoremstyle{definition}
 \newtheorem{definition}[theorem]{Definition}
  \newtheorem{example}[theorem]{Example}
   \newtheorem{remark}[theorem]{Remark}
   
\newcommand{\Naturali}{{\mathbb{N}}}
\newcommand{\Reali}{{\mathbb{R}}}
\newcommand{\Complessi}{{\mathbb{C}}}
\newcommand{\Toro}{{\mathbb{T}}}
\newcommand{\Relativi}{{\mathbb{Z}}}
\newcommand{\HH}{\mathfrak H}
\newcommand{\KK}{\mathfrak K}
\newcommand{\LL}{\mathfrak L}
\newcommand{\as}{\ast_{\sigma}}
\newcommand{\tn}{\vert\hspace{-.3mm}\vert\hspace{-.3mm}\vert}
\def\A{{\cal A}}
\def\B{{\cal B}}
\def\E{{\cal E}}
\def\F{{\cal F}}
\def\H{{\cal H}}
\def\K{{\cal K}}
\def\L{{\cal L}}
\def\N{{\cal N}}
\def\M{{\cal M}}
\def\gM{{\frak M}}
\def\O{{\cal O}}
\def\P{{\cal P}}
\def\S{{\cal S}}
\def\T{{\cal T}}
\def\U{{\cal U}}
\def\V{{\mathcal V}}
\def\qed{\hfill$\square$}

\title{On discrete twisted 
C*-dynamical systems, Hilbert C*-modules and regularity}

\author{Erik B\'edos$^*$,
Roberto Conti$^{**}$ \\}
\date{\today}
\maketitle
\markboth{R. Conti, Erik B\'edos}{
}
\renewcommand{\sectionmark}[1]{}
\begin{abstract} We first give an overview of the basic theory for discrete unital twisted C$^*$-dynamical systems and their covariant representations on Hilbert C$^*$-modules. After introducing the notion of equivariant representations of such systems and their product with covariant representations, we prove a kind of Fell absorption principle saying that the product of an induced regular equivariant representation with a covariant faithful representation is weakly equivalent to an induced regular covariant representation. This principle is the key to our main result, namely that a certain property, formally weaker than Exel's approximation property, ensures that the system is regular, i.e., the associated full and reduced C$^*$-crossed products are canonically isomorphic. 


\end{abstract}

\vfill
\thanks{\noindent $^*$  research partially supported by the Norwegian Research
Council (NFR).\par
}

\newpage

\bigskip

\section{Introduction} 

In a previous work \cite{BeCo2} we have discussed convergence and summation of Fourier series of elements in reduced twisted group $C^*$-algebras associated with discrete groups. As Fourier series may also be defined for elements in the reduced  $C^*$-crossed product $C^*_r(\Sigma)$ of a discrete twisted $C^*$-dynamical system $\Sigma=(A,G, \alpha, \sigma)$, a natural question is how much of our analysis can be transfered to this more general case. 
One possible approach is to consider $C^*_r(\Sigma)$ as the reduced cross sectional algebra of a  Fell bundle over the discrete group $G$ (see \cite{ExLa}). Then 
a result  of R. Exel  \cite{Ex}
for such Fell bundles with the so-called approximation property may be applied to produce certain 
summation processes, 
that may be considered as analogs of the classical Fej{\' e}r summation process. 

In order to construct other types of summation processes for $C_r^*(\Sigma)$, 
we have reached  the conclusion that one should exploit the structure of discrete twisted 
$C^*$-crossed products and their representation theory on Hilbert  $C^*$-modules. 
A feature that is apparently not available in the setting of Fell bundles is the concept of an {\it equivariant representation}  of  $\Sigma$ on a Hilbert $A$-module.  As will be seen in our forthcoming work \cite{BeCo3}, such representations may be used to induce  (completely bounded) multipliers of $C_r^*(\Sigma)$, which in turn are  basic ingredients in the construction of summation processes. 

It seems to us that equivariant representations  of $\Sigma$ have an important role to play in the representation theory of $\Sigma$, complementary to the one played by covariant representations.
For example, one may think of Exel's approximation property for $\Sigma$ as saying that the trivial equivariant representation of $\Sigma$ is weakly contained in the regular equivariant representation of $\Sigma$, hence as a form of amenability of $\Sigma$. As shown by Exel  \cite{Ex} (in the context of Fell bundles; see also \cite{ExNg}), the approximation property implies that $\Sigma$ is regular, i.e., the full crossed product $C^*(\Sigma)$ and $C^*_r(\Sigma)$ are canonically isomorphic, and it is an open problem whether the converse holds or not. 
Besides providing us with some of the tools necessary for our work in \cite{BeCo3}, our aim with this paper is to 
show that one may  weaken Exel's approximation property for $\Sigma$ without loosing regularity. Due to the present lack of examples of systems having the weak approximation property without having Exel's approximation property, it is possible that this weakening is only of a formal character. Nevertheless, in our opininon, it provides a more flexible concept, and we will illustrate this by proving a permanence result that is not obviously true if one sticks to the approximation property.        

The paper is organized as follows. Since the twisted case in its full generality is not covered in the existing literature in an adequate way for our purposes, we give in Sections 2 and 3 an introduction to discrete twisted $C^*$-dynamical systems, their covariant representations on Hilbert C$^*$-modules, and  the associated full and reduced C$^*$-crossed products. Our presentation 
relies on many sources, such as \cite{AD1, AD2, BrOz, Com, Ec, Ex, ExNg, La1, PaRa, PaRa1, ZM, Wi}.

Section 4 is devoted to equivariant representations of discrete twisted $C^*$-dynamical systems and a generalization of Fell's absorption principle. 
After introducing what we mean by an equivariant representation $(\rho, v)$ of $\Sigma$ on a Hilbert $A$-module,
we show that one may
form the product of $(\rho, v)$ with a covariant representation $(\pi, u)$ of $\Sigma$ to obtain another covariant representation $(\rho\dot\otimes\pi, v\dot\otimes u)$. Moreover,  if $(\check{\rho}, \check{v})$ is the regular equivariant representation induced by $(\rho, v)$  and $\pi$ is faithful, then we prove that $(\check{\rho}\dot\otimes\pi, \check{v}\dot\otimes u)$ is weakly equivalent to an induced regular covariant representation (see Theorem \ref{fell}). 

In Section 5 we define the weak approximation property for $\Sigma$. Every equivariant representation induces a regular equivariant representation
and, in a certain sense, the weak approximation property of $\Sigma$ says that the trivial equivariant representation  is weakly contained in some induced regular equivariant representation.
Using our version of Fell's absorption principle,  we show (Theorem \ref{mainreg}) that the weak approximation property implies that $\Sigma$ is regular. 
This result covers most of the known results in this direction (see e.g.\  \cite{AD1, AD2, BrOz, Ex, ExNg, PaRa1, ZM}). 
We also show that under a natural assumption (the existence of an equivariant conditional expectation), the weak approximation property for $\Sigma$ is inherited by a so-called $G$-subsystem. This permanence result has several interesting consequences, and we used it to show that some systems have the approximation property (see Example \ref{exact}).

\medskip  
To avoid many technical details that would obscure our exposition, we will assume that all
$C^*$-algebras in this article are {\it unital}, unless otherwise specified, and $A$ will always denote such a $C^*$-algebra. The diligent reader 
will surely be able to handle
the non-unital case if necessary.

\smallskip Throughout the paper, we will use the following conventions. By a homomorphism of a unital $*$-algebra into a $C^*$-algebra, we will mean a unital $*$-homomorphism.  Isomorphisms between $C^*$-algebras will also be assumed to be $*$-preserving. The group of unitary elements in $A$ will be denoted by $\U(A)$, the center of $A$ by $Z(A)$, while the group of ($*$-preserving) automorphisms of $A$ will be denoted by ${\rm Aut}(A)$. The identity map on $A$ will be denoted by ${\rm id}$ (or ${\rm id}_A$). 

If $\phi_1:A \to B_1$ and $\phi_2:A\to B_2$ are homomorphisms between $C^*$-algebras, we will say that 
$\phi_1$ is {\it weakly contained} in $\phi_2$ (resp.\ is {\it weakly equivalent} to $\phi_2$)   whenever ${\rm \ker}(\phi_2) \subseteq {\rm ker}(\phi_1)$ (resp.\ ${\rm \ker}(\phi_2) = {\rm ker}(\phi_1)$), that is, whenever there exists a homomorphism (resp.\ isomomorphism) $\psi$ from $ \phi_2(A)$ onto $\phi_1(A)$ such that $\psi\circ\phi_2= \phi_1$. 

\smallskip By a Hilbert $C^*$-module, we will always mean a  {\it right}� Hilbert $C^*$-module and follow the notation introduced in \cite{La1}. Especially, all inner products will be assumed to be linear in the second variable, $\L(X, Y)$  will denote the space of all adjointable operators between two Hilbert $C^*$-modules $X$ and $Y$ over a common $C^*$-algebra, and $\L(X) = \L(X,X)$. A representation of $A$ on a Hilbert $C^*$-module $X$ is then a homomorphism from $A$ into the $C^*$-algebra $\L(X)$. If $Z$ is another Hilbert $C^*$-module (possibly over some other $C^*$-algebra), we will let $\pi \otimes \iota : A  \to \L(X \otimes Z)$ denote the amplified representation of $A$ on $X \otimes Z$ given by $(\pi \otimes \iota)(a) = \pi(a) \otimes I_Z$ where the Hilbert $C^*$-module $X \otimes Z$ is the external tensor product of $X$ and $Z$ and $I_Z$ denotes the identity operator on $Z$.

\section{Twisted crossed products by discrete groups}
Throughout this paper, the quadruple $\Sigma = (A, G, \alpha,\sigma)$ will denote a {\it twisted 
(unital, discrete) 
$C^*$-dynamical system}, This means that
$A$ is a $C^*$-algebra with unit $1$, 
$G$ is a discrete group with identity $e$
and $(\alpha,\sigma)$ is a {\it twisted
action} of $G$ on $A$, that is,
$\alpha$ is a map from $G$ into ${\rm Aut}(A)$  
and  $\sigma$ is a map from $G \times G$ into $\, \U(A)$,
satisfying
\begin{align*}
\alpha_g \circ \alpha_h & = {\rm Ad}(\sigma(g, h)) \circ  \alpha_{gh} \\
\sigma(g,h) \sigma(gh,k) & = \alpha_g(\sigma(h,k)) \sigma(g,hk) \\
\sigma(g,e) & = \sigma(e,g) = 1 \ , 
\end{align*}
for all $g,h,k \in G$. Of course,  ${\rm Ad}(v)$ denote here the (inner) automorphism  of $A$ implemented by some unitary $v$ in $A$.

\medskip One can readily deduce from the above relations a number of 
other useful identities, for instance
$$\alpha_e = {\rm id}, \, \sigma(g,g^{-1}) = \alpha_g(\sigma(g^{-1},g))$$
and
$$\alpha^{-1}_g
= \alpha_{g^{-1}} \circ  {\rm Ad}(\sigma(g,g^{-1})^*)
= {\rm Ad}(\sigma(g^{-1},g)^*) \circ  \alpha_{g^{-1}}$$
to quote a few.

\medskip
Note that if 
 $\sigma$ is {\it central}, that is, takes values in the center $Z(A)$ of $A$, then 
$\alpha$ is a homomorphism from $G$ into ${\rm Aut}(A)$, i.e.\ is an ordinary action of $G$ on $A$. The seminal paper of Zeller-Meier \cite{ZM}     
still contains a lot of valuable information on this case. If $\sigma$ is scalar-valued, that is,  takes values in $\Toro$  (identified with $\mathbb{T}\cdot 1$), then $\sigma$ is just a normalized 2-cocycle on $G$, i.e.\ belongs to the second cohomology group $Z^2(G,\mathbb{T})$.  This is especially satisfied when $A=\mathbb{C}$  and we refer e.g.\  to \cite{BeCo2} and references therein for more information on this special case. 
If $\sigma$ is 
trivial, that is, $\sigma(g,h)=1$ for all $g,h \in G$, then $\Sigma$ is an ordinary (untwisted) $C^*$-dynamical system, and  such systems are studied (in their full generality) in several books, the most recent one being \cite{Wi}; see also \cite{Ec} for a nice overview.

\medskip
To each twisted $C^*$-dynamical system  $\Sigma = (A, G, \alpha,\sigma)$ 
one may associate its {\it full twisted crossed product} 
$C^*(\Sigma)$ 
and its {\it reduced twisted crossed product} 
$C^*_r(\Sigma)$ (see \cite{PaRa, PaRa1}).  We will recall below their definitions and some of their basic properties. This can be done without much trouble, at least in the discrete case, by making use of Hilbert $C^*$-modules. This approach is definitely not new, but we could not find it in the literature in a unified form suitable for our purposes in the present paper (and in \cite{BeCo3}).

\medskip

The vector space $C_c(G,A)$ of functions from 
$G$ into $A$ with finite support  becomes a (unital) $*$-algebra, denoted by $C_c(\Sigma)$, when equipped with the  
twisted convolution product and the involution given by:
\begin{equation}\label{convol}
(f_1 \ast f_2)(h) 
= \sum_{g \in G} f_1(g) \, \alpha_g(f_2(g^{-1}h)) \, \sigma(g,g^{-1}h), 
 \end{equation}
\begin{equation}\label{star}
f^*(h) = \sigma(h, h^{-1})^*\, \alpha_h(f(h^{-1}))^*  \ ,
 \end{equation}
 where $\ f_1, f_2, f \in C_c(\Sigma), h \in G \ .$
 
\medskip Hereafter, we will let $a \odot  \delta_g \in C_{c}(G,A)$  denote the function which is 
 0 every- where except at the point $g \in G$, where it takes the value $a \in A$. Obviously, $1\odot\delta_e$ is then the unit of $C_c(\Sigma)$.
 
\medskip By a  {\it covariant homomorphism} of $\Sigma$ we will mean a pair $(\pi,u)$,
where $\pi$ is a 
homomorphism of $A$ into a  
$C^*$-algebra $C$ and  $u$ is a map of $G$ into $\U(C)$, which satisfy
$$u(g)\, u(h) = \pi(\sigma(g,h))\,  u(gh)$$ and
the covariance relation 
\begin{equation}
\pi(\alpha_g(a)) = u(g) \, \pi(a) \, u(g)^* 
\end{equation}
 for all $g,h \in G$, $a�\in A$. 
If $C=\L(X)$ for some Hilbert $C^*$-module $X$, we then say that
  $(\pi, u)$ is a {\it covariant representation} of $\Sigma$ on $X$. 

\medskip There exists a bijective correspondence between covariant homomorphisms
of $\Sigma$ and 
homomorphisms of $C_c(\Sigma)$, 
that associates to each $(\pi,u)$ the unital $*$-homomorphism $\pi \times u$ 
given by
\begin{equation}
(\pi \times u)(f) = \sum_{g \in G} \pi(f(g))\, u(g), \ f \in C_c(\Sigma) \,. 
\end{equation}
The ``integrated form'' $\pi \times u$  satisfies
$(\pi \times u)(a \odot \delta_g)= \pi(a)\, u(g)$ 
for all $a \in A$ and $g \in G$. 

\medskip The $C^*$-algebra $C^*(\Sigma)$ is
the completion of the $*$-algebra $C_c(\Sigma)$ with respect to the $C^*$-norm
\begin{equation}\label{fullnorm}
\|f\|_{*} = \sup \{�\|(\pi \times u)(f)\| \,:\, (\pi,u)\, \,  \text{is a covariant homomorphism of} \, \, \Sigma\}.
\end{equation}

As will soon be explained, equation (\ref{fullnorm}) gives indeed a norm (and not only  a seminorm) on $C_{c}(\Sigma)$,  
and we will  identify $C_c(\Sigma)$ with its canonical copy inside $C^*(\Sigma)$. 
Any  homomorphism $\phi$ from $C_{c}(\Sigma)$ into some $C^*$-algebra extends uniquely to a homomorphism of $C^*(\Sigma)$, still denoted by $\phi$.
Conversely, every homomorphism $\phi$ of $C^*(\Sigma)$  into some $C^*$-algebra $C$ comes from a  homomorphism defined on $C_{c}(\Sigma)$, and we have $\phi = \pi \times u$ where $(\pi,u) $ is the covariant homomorphism of $\Sigma$ into $C$ given by $$\pi(a)= \phi( a \odot \delta_e), \, u(g) = \phi( 1 \odot \delta_g), \, \quad a \in A, \, g \in G\,.$$ 
For example, the identity morphism ${\rm id}_\Sigma: C^*(\Sigma)\to C^*(\Sigma)$ disintegrates as ${\rm id}_\Sigma= i_A\times i_G$ where 
$(i_A, i_G)$ is the {\it universal covariant homomorphism of $\Sigma$} into $C^*(\Sigma)$ given by $$i_A(a) = a \odot \delta_e\, , \, i_G(g) = 1 \odot \delta_g\, , \, \quad a \in A, \, g \in G\,.$$

\medskip 
We next turn our attention to {\it regular} covariant representations of $\Sigma$. Let $Y$ be a Hilbert $B$-module 
and assume $\pi$ is a 
representation of $A$ on $Y$.
We can then form   the  Hilbert $B$-module $Y^G$ given by
\begin{equation}
Y^G 
= \Big\{\xi: G \to Y \ | \ \sum_{g \in G} \big\langle\xi(g),\xi(g)\big\rangle \ 
\mbox{is norm-convergent in  $B$} \Big\} \, 
\end{equation}
endowed with the $B$-valued scalar product 
$$\big\langle\xi,\eta \big\rangle = \sum_{g\in G} \big\langle \xi(g),\eta(g)\big\rangle$$ and the 
natural module right action of $B$ given by
$$(\xi\cdot b )(g) = \xi(g)\, b \,, \quad g \in G\,.$$

The  {\it regular covariant  representation} 
$(\tilde{\pi},  \tilde{\lambda}_\pi)$ of $\Sigma$ on $Y^G$ induced by $\pi$ is then
defined by
\begin{align}
(\tilde{\pi}(a)\xi)(h) & = \pi\big(\alpha_h^{-1}(a)\big)\xi(h) \ , 
\hspace{10ex}  a \in A, \, \xi \in Y^G, \, h \in G, \ \\
(  \tilde{\lambda}_\pi(g)\xi)(h) & = 
\pi\big(\alpha_h^{-1}(\sigma(g,g^{-1}h))\big)\xi(g^{-1}h) \ , 
\quad g,h \in G, \, \xi \in Y^G \, .
\end{align}
It is tedious, but straightforward, to check that $(\tilde{\pi},  \tilde{\lambda}_\pi)$ is indeed a covariant representation of $\Sigma$  on $Y^G$. 

\medskip As a special case, we  consider  $A$ itself as a Hilbert $A$-module in the standard way and let $\ell: A \to \L(A)$  be given  by $\ell(a)(a')= a  a' , \, a, a' \in A$. The regular covariant  representation 
$(\tilde{\ell}, \tilde{\lambda}_\ell)$
  associated to $\ell$
   acts on the Hilbert $A$-module  
   \begin{equation}
A^G 
= \Big\{\xi: G \to A \ | \ \sum_{g \in G} \, \xi(g)^*\,\xi(g) \ 
\mbox{is norm-convergent in  $A$}\Big\}\footnote{$A^G$ is often denoted by $ \ell^2(G, A)$ in the literature.}  \, 
\end{equation}
 in the following way:
\begin{align}
(\tilde{\ell}(a)\xi)(h) & = \alpha_h^{-1}(a) \,\xi(h) \ , 
\hspace{10ex}  a \in A, \, \xi \in A^G, \, h \in G, \ \\
(  \tilde{\lambda}_\ell(g)\xi)(h) & = 
\alpha_h^{-1}(\sigma(g,g^{-1}h))\,\xi(g^{-1}h) \ , 
\quad g,h \in G, \, \xi \in A^G \, .
\end{align}

The homomorphism $\Lambda = \tilde{\ell}\times \tilde{\lambda}_\ell: C_c(\Sigma) \to \L(A^G)$ is easily seen to be faithful. 
Consequently,  $\| \cdot\|_{*}$ is a norm on $C_c(\Sigma)$. Moreover, this allows us to define another $C^*$-norm $\| \cdot \|_r$ on $C_c(\Sigma)$ by setting
$$\|f\|_r = \| \Lambda(f)\|\,, \, \quad f \in C_c(\Sigma)\,.$$ 
The $C^{*}$-completion of
$C_c(\Sigma)$ with respect $\|�\cdot \|_r$ is  denoted by $C^*_r(\Sigma)$.
More concretely, we will often consider  $C^*_r(\Sigma)$ to be the $C^*$-subalgebra of $\L(A^G)$ generated by 
$\Lambda(C_c(\Sigma))$; in other words, we will often  identify $C^*_r(\Sigma)$ with $\Lambda(C^*(\Sigma))$.

\smallskip Now, let us consider again a representation $\pi: A \to \L(Y)$ on a Hilbert $B$-module $Y$. Making use of  the interior tensor product of Hilbert $C^*$-modules (cf.\ \cite{La1}), we can form the Hilbert $B$-module $A^G \otimes_\pi Y$.
We recall that $\pi$ induces  a canonical homomorphism $\pi_*: \L(A^G) \to 
\L(A^G \otimes_\pi Y)$ such that $$\pi_{*}(S)(\xi \dot\otimes \,y) = (S\xi) \dot\otimes \,y\, , \quad S \in \L(A^G), \, \xi \in A^G, \, y \in Y\,.$$ 
The Hilbert $B$-modules $A^G \otimes_\pi Y $ and $Y^G$ are in fact unitarily equivalent. Indeed, the map $U_\pi: A^G \otimes_\pi Y \to Y^G$ determined by
 $$\big[U_\pi(\xi \dot\otimes y)\big](h) = \pi(\xi(h))\, y\, , \quad \xi \in A^G, \, y \in Y, \, h \in G\,  $$ 
 is  easily seen to be a unitary operator in $\L(A^G \otimes_\pi Y, Y^G)$. 
Identifying $A^G \otimes_\pi Y $ and $ Y^G$  via $U_\pi$, one checks that
\begin{align*}
\pi_*(\tilde{\ell}(a)) & 
= \tilde{\pi}(a)\, , \,  a \in A \\
\pi_*(\tilde{\lambda}_\ell(g)) &  
=   \tilde{\lambda}_\pi(g) \, ,\,
g \in G \, .
\end{align*}
It follows that $\, \pi_*\circ \Lambda $  
$= \tilde\pi \times   \tilde{\lambda}_\pi $ on $C^*(\Sigma)$, hence  that $\tilde\pi \times   \tilde{\lambda}_\pi$ is weakly contained in $\Lambda$. 

\medskip If $\pi$ is faithful, then $\pi_*$ is faithful \cite{La1}; hence, in this case,  $\pi_*$ maps $C^*_r(\Sigma)=\Lambda(C^*(\Sigma))$ isomorphically onto $(\tilde{\pi}\times   \tilde{\lambda}_\pi)(C^*(\Sigma))$ 
and $\tilde\pi \times   \tilde{\lambda}_\pi$ is weakly equivalent to $\Lambda$.\footnote{A characterization of the homomorphisms of $C^*(\Sigma)$ which are weakly equivalent to $\Lambda$ will be given in Proposition \ref{CondE}.}
Moreover, choosing $Y$ to be a Hilbert space, 
one hereby recovers the usual definition of the reduced twisted crossed product, that is, $(\tilde{\pi}\times   \tilde{\lambda}_\pi)(C^*(\Sigma))$, and the fact that it does not
depend (up to isomorphism) on the choice of a faithful representation $\pi$ of $A$ on a Hilbert space.

\medskip Some authors prefer to work with other (unitarily equivalent) regular covariant representations of $\Sigma$ on $Y^G$ associated with $\pi:A \to \L(Y)$. For completeness we mention two of them here.

\medskip 
\begin{itemize}
\item[a)] $(\tilde{\pi}\,',  \tilde{\lambda}_\pi')$   is given by \begin{align}
(\tilde{\pi}\,'(a)\xi)(h) & = \pi\big(\alpha_{h^{-1}}(a)\big)\,\xi(h) \,,
\hspace{5ex}  a \in A, \, \xi \in Y^G, \, h \in G, \ \\
(  \tilde{\lambda}_\pi'(g)\xi)(h) & = 
\pi\big(\sigma(h^{-1},g)\big)\,\xi(g^{-1}h) \,, 
\quad \quad    g,h \in G, \, \xi \in Y^G \,.
\end{align}
Letting $S$ be the  operator in $\L(Y^G)$ given by $(S\xi)(g)=\pi(\sigma(g^{-1},g)) \xi(g)$, 

one easily verifies that
$S$ is unitary and $S\, (\tilde{\pi}\times  \tilde{\lambda}_\pi)\,S^*= \tilde{\pi}\,' \times  \tilde{\lambda}_\pi'$.

\medskip 
\item[b)] $(\tilde{\pi}\,'',\tilde{\rho}_\pi)$    is given by 
\begin{align}
(\tilde{\pi}\,''(a)\xi)(h) & = \pi\big(\alpha_{h}(a)\big)\,\xi(h) \,,
\hspace{7ex}  a \in A, \, \xi \in Y^G, \, h \in G, \ \\
(\tilde{\rho}_\pi(g)\xi)(h) & = 
\pi\big(\sigma(h, g)\big)\,\xi(hg) \,, 
 \quad \quad \quad   g,h \in G, \, \xi \in Y^G \,.
\end{align}
Letting $T$ be the  operator in $\L(Y^G)$ given by $(T\xi)(g)=\xi(g^{-1})$, one 
checks without difficulty that
$T$ is an  involutive unitary which satisfies $T \,(\tilde{\pi}\,'\times  \tilde{\lambda}_\pi')\,T= \tilde{\pi}\,'' \times\tilde{\rho}_\pi$.

\end{itemize}

There is another interesting representation of $C^*_r(\Sigma)$  
on a certain Hilbert $A$-module $A^\Sigma$,
which will provide a convenient framework to deal with Fourier analysis on $C^*_r(\Sigma)$.
The module $A^\Sigma$ is defined as follows (cf. \cite[p.\ 302]{AD1} for a similar construction):

\medskip \noindent We set   
$$A^\Sigma= \Big\{\xi: G \to A \ | \ \sum_{g \in G}  \alpha_g^{-1}\big(\xi(g)^*\, \xi(g)\big) \ 
\mbox{is norm-convergent in  $A$} \Big\} \, $$
and endow this vector space with the $A$-valued scalar product 
$$\langle\xi,\eta \rangle_\alpha = \sum_{g\in G} \alpha_g^{-1}\big(\xi(g)^*\eta(g)\big)\,, \quad \xi, \eta \in A^\Sigma$$ and the 
right action of $A$ given by
$$(\xi \times a )(g) = \xi(g)\,\alpha_g(a)  \,, \quad \xi \in A^\Sigma,\,  a \in A, \, g \in G\,.$$
 
Then $A^\Sigma$ becomes a Hilbert $A$-module containing $C_c(G,A)$ as a dense submodule. By construction, the map $J: A^G \to A^\Sigma$ defined by
$$(J\xi)(g)=\alpha_g(\xi(g))\, , \quad \xi \in A^G, \, g \in G\,,$$ is a unitary operator in $\L(A^G, A^\Sigma)$,
with $$(J^*\xi')(g)=\alpha_g^{-1}(\xi'(g))\, , \quad \xi' \in A^\Sigma, \, g \in G\,.$$
We will denote the norm in $A^\Sigma$ by $\|\cdot\|_\alpha$, i.e. we set
$$\|\xi\|_\alpha = \big\| \,\sum_{g\in G} \alpha_g^{-1}\big(\xi(g)^*\xi(g)\big) \,\big\|^{1/2}\,, \quad \xi \in A^\Sigma\,.$$
As $A^G$ and $A^\Sigma$ are unitarily equivalent via $J$, we  obtain a covariant representation $(\ell_\Sigma, \lambda_\Sigma)$ of $\Sigma$ on $A^\Sigma$ by setting $$\ell_\Sigma(a) = J \,\tilde{\ell}(a) \, J^* \, , \quad  \lambda_\Sigma(g) =  J \,\tilde{\lambda}_\ell(g) \, J^* $$
for $a\in A$, $g \in G$. A short computation gives the following expressions:
\begin{equation}\label{LS}(\ell_\Sigma(a)\xi)(h)  = a\, \xi(h)\, , 
\end{equation}
\begin{equation}\label{laS}
(\lambda_\Sigma(g)\xi)(h) = 
\alpha_g(\xi(g^{-1}h))\, \sigma(g, g^{-1}h)
\end{equation}
where $\xi \in A^\Sigma, \, h \in G$.

By construction, the representation $\Lambda_\Sigma= \ell_\Sigma \times \lambda_\Sigma$ of $C^*(\Sigma)$ on $A^\Sigma$  is 
unitarily equivalent to $\Lambda=\tilde{\ell} \times \tilde{\lambda}_\ell$. Moreover, using (\ref{LS}) and (\ref{laS}), one readily sees that it satisfies the nice formula
$$\Lambda_\Sigma(f) \, \xi = f \ast \xi\,,  \quad \, f\in C_c(\Sigma)\,, \, \xi \in A^\Sigma\,, $$
where the convolution $ f \ast \xi$ is defined in the same way as in equation (\ref{convol}). 
  We will henceforth sometimes identify $C^*_r(\Sigma)$
  with $\Lambda_\Sigma(C^*(\Sigma))$ and indicate this by writing  $C^*_r(\Sigma) \subseteq \L(A^\Sigma)$.

\section{Conditional expectations and the Fourier transform}

It is well known that 
there is a canonical faithful conditional expectation $E$ from $ C^*_r(\Sigma)$ onto  the canonical copy of $A$ inside $ C^*_r(\Sigma)$.  This expectation may then be used to define the Fourier coefficients and the Fourier transform of any element of $ C^*_r(\Sigma)$ (see e.g.\ \cite{ZM} and \cite{Bed}).
In fact, these notions are most easily introduced (in the reverse order) by letting $C^*_r(\Sigma)$ act on $A^\Sigma$.  

\medskip We first set $\xi_0 = 1\odot \delta_e \in A^\Sigma$. Then, given $x\in C^*_r(\Sigma) \subseteq \L(A^\Sigma)$, we define its {\it Fourier transform} $\, \widehat{x} \in A^\Sigma$ by $$\widehat{x} = x \, \xi_0\, \,$$
and call $\, \widehat{x}(g)\in A$ the {\it Fourier coefficient} of $\, x $ at $\,g\in G$.

\medskip Clearly, the 
{\it Fourier transform} $\, x \to \widehat{x}$ from $C^*_r(\Sigma)$ into $A^\Sigma$ is  linear.  
Moreover, it is not difficult to verify that it possesses the following properties:

\smallskip \begin{itemize}
\item[(i)]  $\widehat{\Lambda_\Sigma(f)} 
= f$ whenever $f \in C_c(\Sigma)$.

\smallskip  Especially,\,  $\widehat{\ell_\Sigma(a)} = a\odot \delta_e$, 
\, \, $\widehat{\lambda_\Sigma(g)} = 1\odot\delta_g$.

 \item[(ii)] $x \, \xi= \widehat{x} \ast \xi$ whenever $x \in C^*_r(\Sigma), \, \xi \in C_c(G,A) $.\footnote{The  convolution of $ \widehat{x} $ and $\xi$ is defined in the same way as in equation (\ref{convol}). This makes sense as $\xi$ is assumed to have finite support.}

 \item[(iii)] The Fourier transform $\, x \to \widehat{x}$ is injective. 

\item[(iv)] For $ x \in C^*_r(\Sigma) $ we have 
$\, \|\widehat{x}\|_\infty \leq \|\widehat{x}\|_\alpha \leq \|x\|$ 

where $\, \|\widehat{x}\|_\infty = \sup_{g\in G} \|\widehat{x}(g)\|$, by definition.

\item[(v)] $\widehat{x\,y} = \widehat{x}\, \ast \, \widehat{y}$ whenever $x \in  C^*_r(\Sigma), \,y \in \Lambda(C_c(\Sigma))$.\footnote{\,This probably also holds when $ y \in C^*_r(\Sigma)$, but an extra effort is then needed to show that the convolution product makes sense in this case.}

\item[(vi)] $\widehat{x^*} = \widehat{x}^{\,*}$ whenever $x \in  C^*_r(\Sigma)$ (and $\widehat{x}^{\,*}$ is defined as in eq.\ (\ref{star})).

\end{itemize}

\medskip Next, we  define a map $E_A: C^*_r(\Sigma) \to A$ by $$E_A(x) = \widehat{x}(e)\,.$$
Clearly, $E_A$ is linear, bounded with norm one and satisfies $E_A(x)= \langle \xi_0 \,, \, x\,\xi_0 \rangle_\alpha \,$. 

\medskip Moreover, the following properties are almost immediate:
\begin{itemize}
\item[(i)] $E_A(\Lambda_\Sigma(f)) = f(e)\, , \, f\in C_c(\Sigma)$.

\smallskip  Especially, $E_A(\ell_\Sigma(a))= a$\,  and  $E_A(\lambda_\Sigma(g)) = 0$  when $g \neq e\,.$
\item[(ii)] $E_A( x\, \lambda_\Sigma(g)^*) = \widehat{x}(g)\, , \quad g \in G$.
\item[(iii)] $E_A(x^*x) = \|\widehat{x}\|_\alpha^{\,2}$ \, , \quad $x \in C^*_r(\Sigma)$.
\item[(iv)] $E_A\big(\lambda_\Sigma(g)\, x \,\lambda_\Sigma(g)^*\big) = \alpha_g(E_A(x))\, $, \quad $g \in G$\,, \,$x \in C^*_r(\Sigma)$.  
\end{itemize}

Letting $E_\Sigma: C^*_r(\Sigma) \to \ell_\Sigma(A)$ be the linear map given by $E_\Sigma= \ell_\Sigma\circ E_A\,$, we get a norm one projection onto  $\ell_\Sigma(A)$, hence a conditional expectation, which is faithful (using (iii) and the injectivity of $x \to \widehat{x}$). Moreover, using (i), we see that it satisfies $E_\Sigma(\lambda_\Sigma(g)) = 0$ whenever  $g \in G\,,\,  g \neq e\,.$ 

\medskip The existence of such a conditional expectation characterizes (up to weak equivalence) $\Lambda_\Sigma$, and thereby also $\Lambda$  (see \cite[Theoreme 4.22]{ZM} for a similar result in the central case):

\begin{proposition}\label{CondE} Let $\phi=\pi\,\times\, u$ be a homomorphism of $C^*(\Sigma)$ into some $C^*$-algebra and set $B= \phi (C^*(\Sigma))$. 
Then the following conditions are equivalent:
\begin{itemize}
\item[(i)] $\phi$ is weakly equivalent to $\Lambda_\Sigma$\,.
\item[(ii)] $\phi$ is weakly equivalent to $\Lambda$\,.
\item[(iii)]  $\pi$ is injective and there exists a faithful conditional expectation $F$ from $B$ onto $\pi(A)$ satisfying $F(u(g)) =0$ for  all $g \in G\,,\,  g \neq e\,.$
\end{itemize}
Especially, if (iii) holds, then $B$ is isomorphic to $C^*_r(\Sigma)$.
\end{proposition}

\begin{proof} Since $\Lambda$  is weakly equivalent to $\Lambda_\Sigma$, it is clear that $(i)$ is equivalent to $(ii)$. We will show that $(i)$ is equivalent to $(iii)$. 

Assume that $(i)$ holds. Then there exists an isomorphism $\psi$ from $C_r^*(\Sigma)$ onto $B$ satisfying $\psi \circ \Lambda_\Sigma= \phi$.  Especially, $\psi\circ \ell_\Sigma= \pi$. As $\ell_\Sigma$ is injective, it follows that $\pi$ is injective. Moreover we can define  a linear map $F: B \to \pi(A)$ by $F= \pi \circ E_A \circ \psi^{-1}$. Clearly, $F$ is bounded with norm one. Since $\pi$ is injective and $E_A$ is faithful, $F$ is faithful. Further, we have
$$ F((\phi)(f)) = \pi((E_A( \Lambda_\Sigma(f))) = \pi(f(e)) \, , \quad f\in C_c(\Sigma)\,.$$ 
This means that $F(\pi(a)) = \pi(a)\,$ for all $a\in A$ and $F(u(g))= 0\, $ for all $\,g \in G, \, g\neq e$. It follows that $F$ is a norm one projection onto $\pi(A)$, hence a conditional expectation satisfying $(iii)$.
 
  \medskip Conversely, assume that $(iii)$ holds. Define a linear map $\mathcal{E}: C^*(\Sigma) \to A$ by $\mathcal{E}= E_A\circ\Lambda_\Sigma$. Clearly, $\mathcal{E}$ is bounded with norm one.
 
 Let $f\in C_c(\Sigma)$. Then $\pi(\mathcal{E}(f)) = \pi(E_A(\Lambda_\Sigma(f))) = \pi(f(e))$.  On the other hand, $(iii)$ gives that $F((\pi\times u)(f)) = \pi(f(e))$. As $C_c(\Sigma)$ is dense in $C^*(\Sigma)$, this implies that 
 \begin{equation}\label{E}
 F\circ \phi= \pi \circ \mathcal{E}\,.
 \end{equation} 
 Since $\pi$ is injective (by assumption) and $F$ is onto $\pi(A)$, this means that
 \begin{equation}\label{F}
 \mathcal{E}= \pi^{-1} \circ F\circ \phi\,.
 \end{equation}  
 Set $\mathcal{I}= {\rm ker }(\phi)\, , \, \mathcal{J}= {\rm ker }(\Lambda_\Sigma)$. 
 
 Consider $x \in \mathcal{I}^{+}.$
 Then equation (\ref{F}) gives that $\mathcal{E}(x)=0$, hence that $\Lambda_\Sigma(x)=0$ (since $E_A$ is faithful). Thus, $x  \in \mathcal{J}^{+}$.
  
 Conversely, let $x \in \mathcal{J}^{+}.$ Then  $\mathcal{E}(x)=E_A(\Lambda_\Sigma(x))= 0$ , so equation (\ref{E}) gives that $F(\phi(x))=0$. But $\phi(x)$ is positive and $F$ is faithful (since (iii) is assumed to hold), so $\phi(x)=0$, and it follows that $x\in \mathcal{I}^{+}$. 
 
 We have thereby shown that $\mathcal{I}^{+}= \mathcal{J}^{+}$, so $\mathcal{I} = \mathcal{J}$, which means that $(i)$ holds. 
\end{proof}

Whenever a homomorphism $\phi$ as in Proposition \ref{CondE} satisfy condition $(i)$,  and therefore also $(ii)$ and $(iii)$, we will say that $\phi$ is a {\it regular homomorphism of $C^*(\Sigma)$} and call $B=\phi(C^*(\Sigma))$ a {\it reduced $C^*$-crossed product associated with $\Sigma$}. 

Letting $\psi$ denote the isomorphism from $C_r^*(\Sigma)$ onto $B$ satisfying $\psi \circ \Lambda_\Sigma= \phi $,  we may then define the Fourier transform $y \to \widehat{y}$ from $B$ into $A^\Sigma$ by setting $$\widehat{y} = \widehat{\psi^{-1}(y)}\,, \quad b \in B\,.$$
It is then easy, but informative to check that $$\widehat{y}(g) = (\pi^{-1}\circ F)(y\,u(g)^*)\,, $$
that is, \,
$\pi(\widehat{y}(g))= F(y\,u(g)^*)$ for all $g \in G$. Thereby one recovers  the "usual" way to define the Fourier coefficients of elements in a reduced $C^*$-crossed product.  

Moreover, when $y \in B$, 
the formal sum 
$$\sum_{g \in G} \pi(\widehat{y}(g))\, u(g)$$ is called the {\it Fourier series} of $y$. 
If $y= \phi(f)$ for some $f \in C_c(\Sigma)$ with finite support $K$, then $\widehat{y}= f$, so  the Fourier series of $y$ is just a finite sum over $K$ and its sum is equal to $y$. 
However, as is well known, this series
is in general not necessarily norm-convergent in $B$.
We will study convergence and summation of such Fourier series in a forthcoming paper \cite{BeCo3}.

\section{On Fell's absorption principle and equivariant representations}\label{equivcorr}

Throughout this section, we suppose that $\Sigma=(A, G, \alpha, \sigma)$ is given. 
The classical  Fell's absorption principle for unitary representations of a group (see \cite{Di}) may be generalized 
to $C^*$-dynamical systems in several ways. 

One version of Fell's absorption principle is as follows (see \cite[Proposition 4.1.7]{BrOz} for the untwisted case).

\begin{proposition}\label{fell2}
Let $(\pi, u)$ be a covariant representation of $\Sigma$ on a Hilbert $B$-module $Y$ and let $\lambda$ denote the left regular representation of $G$ on $\ell^2(G)$. 

\smallskip Then $(\pi \otimes \iota\, , u \otimes \lambda) $ is a covariant representation of $\Sigma$ on $Y \otimes \ell^2(G)\simeq Y^G$ and $ (\pi \otimes \iota) \times (u \otimes \lambda)$ is unitarily equivalent to $\tilde{\pi} \times \tilde{\lambda}_\pi$.
\end{proposition}
\begin{proof} Similar to the proof given in \cite{BrOz}. See also Example \ref{fell2again}. 
\end{proof}

We will need a more sophisticated version of Fell's absorption principle, which will rely on the following concept.
\begin{definition} By an 
{\it equivariant representation} 
of $\Sigma= (A, G, \alpha, \sigma)$ 
on
a Hilbert $A$-module $X$ 
we will mean  a pair $(\rho, v)$ where 
 $\rho : A \to \L(X)$ is a representation of $A$ on $X$ and  $v$ is a map from $G$ into the group $\mathcal{I}(X)$ consisting of all $\mathbb{C}$-linear, invertible, bounded maps from $X$ into itself, which  satisfy:
\begin{itemize}
\item[(i)]  \quad $\rho(\alpha_g(a))  = v(g) \, \rho(a) \, v(g)^{-1}\, , \quad  \quad g\in G\,, \,  a \in A$
\item[(ii)]  \quad$v(g)\, v(h)  = {\rm ad}_\rho(\sigma(g,h)) \, v(gh) \, , \quad   \quad g, h \in G$
\item[(iii)] \quad $\alpha_g\big(\langle x \, ,\, x' \rangle\big)   = \langle v(g) x\, ,\, v(g) x' \rangle\, , \quad  \quad g\in G\, , \, x,\,  x' \in X\,$ 
\item[(iv)]  \quad$v(g)(x \cdot a)  = (v(g) x)\cdot \alpha_g(a)\, , 
\quad  \quad \, g \in G,\, x\in X,\, a \in A$. 
\end{itemize}
In (ii) above, $ {\rm ad}_\rho(\sigma(g,h)) \in \mathcal{I}(X) $ is defined by
$${\rm ad}_\rho(\sigma(g,h)) \,x = \big(\rho(\sigma(g,h))\, x \big)\cdot \sigma(g,h)^* \, , \quad g, h \in G,\, x \in X. $$
\end{definition}

\begin{remark} Assume that $\sigma$ is trivial. Then (ii) just says that $g \to v(g)$ is a group homomorphism. Such a homomorphism is called an $\alpha$-equivariant action of $G$ on $X$ whenever (iii) and (iv) hold  (see \cite{Com, AD1}); moreover, if (i) also holds, then $v$ is said to be covariant with $\rho$ (relative to $\alpha$). When $\sigma$ is not trivial, this terminology can not be carried over verbatimly, because both $\rho$ and $v$ are involved in (ii). 
\end{remark}

\begin{remark} Note that the Hilbert $A$-module $X$ above becomes an $A$--$A$ bimodule\footnote{sometimes called a $C^*$-correspondence over $A$.} when the left action of $A$ on $X$ is defined by  $$a�\cdot x = \rho(a)x \,, \quad \, a \in A,\, x\in X\,.$$ Then we have  $${\rm ad}_\rho(\sigma(g,h)) \,x = \sigma(g,h) \cdot x \cdot \sigma(g,h)^* \, , \quad g, h \in G,\, x \in X, \,  \, $$
explaining our choice of notation. 
To match the terminology used in \cite{EKQR} when $\sigma$ is trivial,  we could have said that $v$ is an equivariant $(\alpha, \sigma)$-$(\alpha, \sigma)$ 
(or $\Sigma$-) compatible action of $G$ on the Hilbert $A$--$A$ bimodule $X$ when (i)-(iv) are satisfied.  
 
 \smallskip As a bimodule, $X$ has a {\it central part} $Z_X$, given by  $$Z_X= \{�z \in X\mid a\cdot z = z\cdot a \, \, \text{for all} \, \, a \in A\}\,.$$ 
 It is straightforward to check that $Z_X$ is a closed subspace of $X$  which is invariant under $v$, that is, $v(g)\, z \in Z_X$ for all $\, g\in G$ and $\, z \in Z_X$, and invariant under the left and right actions of $Z(A)$. Note also that $\langle z, z'\rangle \in Z(A)$ for all $z, z' \in Z_X$, since for all $a\in A$, we  have:
 $$ \langle z, z'\rangle a= \langle z, z'\cdot a\rangle=\langle z, a\cdot z'\rangle = \langle a^* \cdot z, z'\rangle = 
 \langle z\cdot a^*, z'\rangle = a \langle z, z'\rangle\,.$$
 This means that $Z_X$ becomes a Hilbert $Z(A)$-module. We may then let $\rho': Z(A) \to \L(Z_X)$ and $v':G \to 
 \mathcal{I}(Z_X)$ be defined by $$\rho'(c)z= c\cdot z = \rho(c)z\, , \, v'(g)z= v(g) z$$
 for all $c \in Z(A), z\in Z_X, g \in G$. 

Now, each automorphism $\alpha_g$ restricts to an automorphism $\alpha'_g$ of  $Z(A)$. As $$\ \alpha'_g \circ \alpha'_h = {\rm Ad}(\sigma(g,h)) \circ \alpha'_{gh} = \alpha'_{gh}$$ for every $g,h \in G$,  $\alpha'$ gives an (untwisted) action of $G$ on $Z(A)$. 
Letting $\Sigma'= (Z(A), G, \alpha', 1)$ denote the associated system\footnote{If $\sigma$ is central, one can also consider the system $(Z(A), G, \alpha', \sigma)$.
}, it is then clear that the following holds:
\end{remark}

 \begin{proposition} \label{centralmod} If $(\rho, v)$ is an equivariant representation of $\Sigma$ on an Hilbert $A$-module $X$, then  $(\rho', v')$ is an equivariant representation of $\Sigma'$ on the Hilbert $Z(A)$-module $Z_X$.
\end{proposition}

\begin{example} The  {\it trivial equivariant representation} of $\Sigma$ is the equivariant representation $(\ell, \alpha)$ of $\Sigma$ on the $A$-module $A$ 
where $\ell: A \to \L(A)$ is defined as before (i.e. $\ell(a)\, a' = a\, a'$); as ${\rm Aut}(A) \subseteq \mathcal{I}(A)$, we have $\alpha: G \to  \mathcal{I}(A)$, so this definition makes sense. Trivially, we have $Z_A=Z(A)$ and $(\ell', \alpha')$ is the trivial equivariant representation of $\Sigma'$ on $Z(A)$.  
\end{example}

\begin{example} The {\it regular equivariant representation} of $\Sigma$ is the equivariant representation $(\check{\ell}, \check{\alpha})$  on $A^G$ defined by
 $$(\check{\ell}(a)\, \xi)(h) = a\, \xi(h)$$
$$(\check{\alpha}(g)\,\xi)(h) = \alpha_g(\xi(g^{-1}h))$$
where $a \in A, \, \xi \in A^G, \, g, h \in G$. 
The central part of $A^G$ is $$Z(A)^G = \{ \xi \in A^G \mid \xi(g) \in Z(A) \, \, \text{for all} \, \, g \in G\}\,.$$
Moreover, $(\check{\ell}{\, '}, \check{\alpha}\, ')$ is the regular equivariant representation of $\Sigma'$ on $Z(A)^G$.

\end{example}

The regular 
equivariant representation 
is induced from the trivial one according to the construction described in the next example.

\begin{example} Let $(\rho,v)$ be an equivariant 
representation of $\Sigma$ on a Hilbert $A$-module $X$.
The {\it induced regular  equivariant 
representation}
 $(\check{\rho},\check{v})$ 
of $\Sigma$ on $X^G$ associated with $(\rho,v)$ is defined by
\begin{align*}
(\check{\rho}(a)\xi)(h) & = \rho(a)\xi(h), \\
(\check{v}(g)\xi)(h) & = v(g)\xi(g^{-1}h) \ ,
\end{align*}
for all $a \in A, \, \xi \in X^G,\, g,h \in G$. We leave it as an exercise to verify that this gives an equivariant representation, the central part of $X^G$ is  $Z_X^G = \{ \xi \in X^G \mid \xi(g) \in Z_X \, \, \text{for all} \, \, g \in G\}$ and $(\check{\rho}{\, '}, \check{v}\, ')$ is the induced regular equivariant representation of $\Sigma'$ on $Z_X^G=(Z_X)^G$ associated with $(\rho', v')$.
\end{example}

\medskip {\it From now on, and throughout this section, we consider  an equivariant 
representation $(\rho,v)$ of $\Sigma$ on a Hilbert $A$-module $X$ and a covariant representation $(\pi,u)$ of $\Sigma$ on a Hilbert 
$B$-module $Y$}. 

\smallskip We may then form the Hilbert $B$-module $X\otimes_\pi Y$ and the canonical homomorphism $\pi_{*}: \L(X) \to \L(X\otimes_\pi Y)$, providing us with the representation $\pi_{*} \circ \rho$ of $A$ on $X\otimes_\pi Y$. Moreover, we have: 

\begin{lemma}  For each $g\in G$ there exists a unitary operator $w(g)$ on $X\otimes_\pi Y$ which satisfies
$$ w(g)(x\dot\otimes\, y) = v(g) x \, \dot\otimes \, u(g)y\, , \quad x \in X\, , \, y \in Y\,.$$ 
The resulting pair $(\pi_{*} \circ \rho\, ,\, w)$ is then a covariant representation of $\Sigma$ on $X\otimes_\pi Y$.
\end{lemma}
\begin{proof} Left to the reader as a routine exercise.
\end{proof} 
We will denote the pair $(\pi_{*} \circ \rho\, ,\, w)$ by $(\rho\dot\otimes\pi\, , \, v\dot\otimes u)$ and call it {\it the product of $(\rho,v)$ with $(\pi,u)$}. 
If for example $(\rho, v)= (\ell,\alpha)$ is the trivial equivariant representation of $\Sigma$ on $A$, then, as is well known \cite{La1}, $A\otimes_\pi Y$ is unitarily equivalent to $Y$, and one may check that  $(\ell\dot\otimes\pi \, , \, \alpha \dot\otimes u)$ corresponds to $(\pi, u)$ under this identification. 

A more interesting interesting case occurs when we consider an induced regular equivariant representation. We will need the following:

\begin{lemma}\label{W}
There exists a unitary operator  $W \in \L(X^G \otimes_\pi Y)$ which satisfies
\begin{equation}\label{eqW}
W\big((x \odot \delta_g ) \dot\otimes \, y\big) 
= ( v(g)x \odot \delta_g) \, \dot\otimes \, u(g)y\, , 
 \end{equation}
for all $\ g \in G, x \in X, y \in Y \,.$ 

Here $x \odot \delta_g $ means the element in $C_c(G,X) \subseteq X^G$ which takes the value $x$ at $g$ and is zero otherwise. 
\end{lemma} 
\begin{proof} We will first define $W$ on the dense subspace $Z$ of $X^G \otimes_\pi Y$ consisting of the span of elements of the form $\xi \dot\otimes y\, ,$ where $ \xi \in C_c(G,X), \, y \in Y$. 
For $z = \sum_{i=1}^{n} \xi_i \dot\otimes y_i \in Z$, we set
$$Wz = \sum_{i=1}^{n} \sum_{g \in G} \big(v(g)\xi_i(g) \odot \delta_g\big) \dot\otimes \, u(g)y_i$$
where the sum over $G$ is actually a finite sum over the union of the supports of the $\xi_i$'s.
Then, if $z' = \sum_{j=1}^{m} \xi'_j \dot\otimes y'_j \in Z$, we have
\begin{align*}
\langle Wz, \, Wz' \rangle & = \sum_{i,j, g, h} \Big\langle \, \big(v(g)\xi_i(g) \odot \delta_g\big)\, \dot\otimes \,u(g)y_i\, , \, \big(v(h)\xi'_j(h) \odot \delta_h\big) \, \dot\otimes \, u(h)y'_j \Big\rangle \\
& = \sum_{i,j, g, h} \Big\langle \,  u(g)y_i\, , \, \pi\Big( \big\langle v(g)\xi_i(g) \odot \delta_g, v(h)\xi'_j(h) \odot \delta_h\big\rangle\Big) \, u(h)y'_j \Big\rangle \\
& = \sum_{i,j, g} \Big\langle \,  u(g)y_i\, , \, \pi\Big( \big\langle v(g)\xi_i(g), v(g)\xi'_j(g) \big\rangle\Big) \, u(g)y'_j \Big\rangle\\
 & = \sum_{i,j, g} \Big\langle \,  u(g)y_i\, , \, \pi\Big( \alpha_g \big(\big\langle \xi_i(g), \xi'_j(g) \big\rangle \big)\Big) \, u(g)y'_j \Big\rangle\\
 & = \sum_{i,j, g} \Big\langle \,  u(g)y_i\, , \, u(g)\, \pi\Big(\big\langle \xi_i(g), \xi'_j(g) \big\rangle\Big) \,y'_j \Big\rangle\\
& = \sum_{i,j, g} \Big\langle \,  y_i\, , \, \pi\Big(\big\langle \xi_i(g), \xi'_j(g) \big\rangle\Big) \,y'_j \Big\rangle = 
\sum_{i,j} \big\langle \,  y_i\, , \, \pi\big(\big\langle \xi_i, \xi'_j \big\rangle\big) \,y'_j \big\rangle\\
& = \sum_{i,j} \big\langle \, \xi_i \dot\otimes y_i\, , \, \xi'_j \dot\otimes y'_j \big\rangle = \langle z\, , \, z' \rangle\,.
\end{align*}
It follows that $W$ is a well defined isometry from $Z$ into itself, which therefore extends to an isometry from $X^G \otimes_\pi Y$ into itself. Moreover, $W$ is easily seen to be $B$-linear, and it satisfies equation (\ref{eqW}) by definition. 

Now, a similar computation shows that 
there exists a $B$-linear isometry $W'$ from $X^G \otimes_\pi Y$ into itself which satisfies \begin{equation}
W' \big((x \odot \delta_g ) \dot\otimes y\big) 
= ( v(g)^{-1}x \odot \delta_g) \, \dot\otimes \, u(g)^*y\, , 
 \end{equation}
 for all $\ g \in G, x \in X, y \in Y \,.$ But then it readily follows that $W$ and $W'$ are inverses of each other. Hence, $W$ is unitary (cf. \cite{La1}).
\end{proof}
The following theorem reduces to Fell's classical absorption principle when $A=\mathbb{C}$ and $\sigma$ is trivial. 

\begin{theorem} \label{fell} Consider the induced equivariant representation $(\check{\rho}, \check{v})$ on $X^G $ and let  $\pi'_{*}:\L(X^G) \to \L(X^G \otimes_\pi Y)$ denote the  canonical homomorphism associated with $\pi$, so that 
$\check{\rho}\dot\otimes\pi= \pi'_{*}\circ \check{\rho} : A  \to \L(X^G \otimes_\pi Y)$.  

\smallskip Then the homomorphisms  
$(\check{\rho}\dot\otimes\pi) \times (\check{v} \dot\otimes u)$ and $\pi'_{*} \circ (\tilde{\rho} \times \tilde{\lambda}_\rho)$ 
are unitarily equivalent.

\smallskip Especially, it follows that  $(\check{\rho}\dot\otimes\pi) \times (\check{v} \dot\otimes u)$ and $\tilde{\rho} \times \tilde{\lambda}_\rho$ are weakly equivalent whenever $\pi$ is faithful. 

\smallskip Hence, $(\check{\rho}\dot\otimes\pi) \times (\check{v} \dot\otimes u)$ is regular whenever both $\pi$ and $\rho$ are  faithful.

\end{theorem}

\begin{proof} Let $a \in A, \, g \in G$ be fixed.  We first note that for all $x\in X, \, h \in G$, we have
$$\tilde{\rho}(a)(x \odot \delta_h) =  
\big(\rho(\alpha_h^{-1}(a))\,x \big) \odot \delta_h\, , \quad \check{\rho}(a)(x \odot \delta_h) =  
\rho(a)\,x  \odot \delta_h$$
and
$$\tilde{\lambda}_\rho(g)(x\odot \delta_h) 
=  \big(\rho(\alpha_{gh}^{-1}(\sigma(g,h)))\, x\big)  \odot \delta_{gh}\, , \quad \check{v}(g)(x\odot \delta_h) 
=  v(g) x  \odot \delta_{gh} \,.$$
Letting $W$ be the unitary operator defined in Lemma \ref{W}, we therefore get
\begin{align*}
\big(W \, \pi'_{*}\big(\tilde\rho(a)&\tilde{\lambda}_\rho(g)\big) \, W^*\big) \, 
\big((x \odot \delta_h) \dot\otimes \,y\big) \\
& = \big(W \, \pi'_{*}\big(\tilde\rho(a)\tilde{\lambda}_\rho(g)\big)\big)  \big((  
v(h)^{-1}x \odot \delta_h) \dot\otimes \,u(h)^* y )\\
& = W \Big( \big(\tilde\rho(a)\tilde{\lambda}_\rho(g) (v(h)^{-1}x \odot \delta_h)\big)
\dot\otimes \, u(h)^{*} y \Big) \\
& = W \Big( \big( \rho\big(\alpha_{gh}^{-1}(a \,\sigma(g,h))\big)
v(h)^{-1} x \odot \delta_{gh}\big) \dot\otimes \,  u(h)^{*} y \Big) \\
& =  \Big(v(gh) \rho\big(\alpha_{gh}^{-1}(a \,\sigma(g,h))\big)
v(h)^{-1} x \odot \delta_{gh}\Big) \dot\otimes \, u(gh)u(h)^{*} y \\
& =   \Big(\rho(a \, \sigma(g,h)) v(gh)v(h)^{-1}x) \odot \delta_{gh}\Big)
\dot\otimes \, \pi(\sigma(g,h)^*)u(g)y \\
& = \Big(\big(\big(\rho(a)  \rho(\sigma(g,h))
v(gh)v(h)^{-1}x\big)\cdot \sigma(g,h)^*\big) \odot \delta_{gh}\Big)  \dot\otimes \, u(g)y \\
& = \Big(\big(\rho(a) \big( {\rm ad}_\rho(\sigma(g,h))
v(gh)v(h)^{-1}x\big) \big) \odot \delta_{gh}\Big)  \dot\otimes \, u(g)y \\
& = \Big(\big(\rho(a) \big( v(g)
v(h)v(h)^{-1}x\big) \big) \odot \delta_{gh}\Big)  \dot\otimes \, u(g)y \\
& = \Big(\rho(a) v(g)x \odot \delta_{gh}\Big)  \dot\otimes \, u(g)y \\
& = \Big(\check\rho(a) \check{v}(g)(x \odot \delta_{h})\Big)  \dot\otimes \, u(g)y \ 
\end{align*}
for each $x\in X, \, h \in G$ and $y \in Y$. By a density argument, we get
\begin{equation} \label{W-eq}
\big(W \, \pi'_{*}\big(\tilde\rho(a)\tilde{\lambda}_\rho(g)\big) \, W^*\big)
(\eta \dot\otimes\, y) = \check\rho(a)\check{v}(g)\eta \dot\otimes \, u(g)y  
\end{equation}
for all $\eta \in X^G$ and $y \in Y$, which in turn gives $$W \,\big( \pi'_{*}\circ(\tilde\rho \times \tilde{\lambda}_\rho)(f) \big)\, W^*
= \big((\check\rho\dot\otimes\pi) \times (\check{v}\dot\otimes u)\big)(f)$$
for all $f\in C_c(\Sigma)$. The first assertion clearly follows. If $\pi$ is faithful, then $\pi'_{*}$ is also faithful. On the other hand,  $\tilde\rho \times \tilde{\lambda}_\rho$ is weakly equivalent to $\Lambda$ when $\rho$ is faithful. Hence, both statements in the final assertion
are a consequence of the first assertion. 
\end{proof}

\begin{example}\label{fell2again}
Let us apply Theorem \ref{fell} with $(\rho, v)=(\ell,\alpha)$. This gives $$(\check{\ell}\dot\otimes\pi) \times (\check{\alpha} \dot\otimes u) \simeq  \pi'_{*} \circ (\tilde{\ell} \times \tilde{\lambda}_\ell)= \pi'_{*} \circ \Lambda \simeq \tilde{\pi} \times \tilde{\lambda}_\pi\,.$$ 
Now, as observed in Section 2, $A^G \otimes_\pi Y$ is unitarily equivalent to $Y^G$, hence to $Y \otimes \ell^2(G)$. Under this identification, one may check that $(\check{\ell}\dot\otimes\pi) \times (\check{\alpha} \dot\otimes u)$ corresponds to $(\pi \otimes \iota) \times (u \otimes \lambda)$. Thus we get that $(\pi \otimes \iota) \times (u \otimes\lambda) \simeq \tilde{\pi} \times \tilde{\lambda}_\pi$, thereby recovering Proposition \ref{fell2}.  
\end{example}

The Fell's absorption principle described in Theorem \ref{fell} may 
be used to construct certain completely bounded maps, and  
this procedure will be useful to us in the next section.

\begin{proposition}\label{Phi-prop} 
Let $\xi, \eta \in X^G$.  Then there exists  a (unique) completely bounded linear map
\begin{equation}
\Phi 
: \L(X^G) \to \L(Y)
\end{equation}
which satisfies $\|\Phi\| \leq \| \Phi \|_{\rm cb}  \leq \|\xi\| \, \|\eta\|$ and
\begin{equation}\label{Phi}
\Phi\big(\tilde{\rho}(a)\tilde{\lambda}_{\rho}(g)\big) 
= \pi\big(\langle \xi\, ,\, \check{\rho}(a)\check{v}(g)\eta \rangle\big)\, u(g)
\end{equation}
for each $\, a \in A, \, g \in G\,.$ 

\medskip If $\eta=\xi$, then $\Phi$ is completely positive and $\|\Phi\|= \|\xi\|^2$.
\end{proposition}

\begin{proof} We use the same notation as in Theorem \ref{fell} and its proof. 
For $\xi \in X^G$,   
we also let $\theta_\xi \in \L(Y,X^G \otimes_\pi Y)$ be defined as in
\cite{La1}, that is,
$\theta_\xi(y) = \xi \dot\otimes \, y\, ,  y \in Y\,.$
Consider the linear map $\Phi : \L(X^G)
\to \L(Y)$  given by
\begin{equation}
\Phi(\cdot)=\theta_\xi^* \, W\, \pi'_{*}(\cdot)\, W^* \,\theta_\eta 
\end{equation}
It is well known that such a map is completely bounded (see e.g.\ \cite{Pau} or \cite{Pis}), with 
$$\|\Phi\|_{cb}\leq \|\theta_{\xi}^* \, W\|\,  \| W^*\theta_{\eta}\|=  \|\xi\| \, \|\eta\|\, .$$
If $\eta=\xi$, then $\Phi(\cdot)=(\theta_\xi^* \, W)\, \pi'_{*}(\cdot)\, (\theta_\xi^* W)^*$ becomes completely positive and satisfy
$$\|\Phi\|= \|\Phi\|_{cb}= \|\Phi(I)\| = \|\xi\|^2\,.$$

Hence, it remains only to show that (\ref{Phi}) holds. So let $a \in A, \, g \in G$. 
Using Theorem \ref{fell} (see equation (\ref{W-eq})), we get
\begin{align*}
\Phi(\tilde\rho(a)\tilde{\lambda}_\rho(g))y &= \big(\theta_\xi^* \, W \pi'_{*}\big(\tilde\rho(a)\tilde{\lambda}_\rho(g)\big)W^*\theta_\eta \big)y \\
&=\big(\theta_\xi^* \,W \pi'_{*}\big(\tilde\rho(a)\tilde{\lambda}_\rho(g)\big)W^*\big) (\eta \dot\otimes y)\\
& = \theta_\xi^*\, \big(\check\rho(a)\check{v}(g)\eta \, \dot\otimes \, u(g)y \big) \\
& = \pi\big(\langle \xi,\check\rho(a)\check{v}(g)\eta\rangle\big)u(g)y \ 
\end{align*}
for each $y \in Y$, as desired.
 \end{proof}
 
 We will see in \cite{BeCo3} that there exist other versions of Fell's absorption principle, and that these may be used to construct (completely bounded) "multipliers" from $C_r^*(\Sigma)$ into itself.

\section{Regularity}

Following the terminology introduced in
\cite{QS}, we will say that $\Sigma=(A,G, \alpha, \sigma)$ is {\it regular} if the canonical homomorphism 
$\Lambda : C^*(\Sigma) \to C^*_r(\Sigma)$ is injective, i.e., $\Lambda$ is an 
isomorphism. Taking into account Proposition \ref{CondE}, one can easily check that $\Sigma$ is regular if and only if some (resp.\ every) regular homomorphism of $C^*(\Sigma)$ is an isomorphism, if and only if some faithful homomorphism (resp.\ every homomorphism) of $C^*(\Sigma)$ is weakly contained in some regular homomorphism of $C^*(\Sigma)$. 

\medskip It is well known that $\Sigma$ is regular whenever $G$ is amenable (see \cite{ZM, PaRa1}).
Some more general conditions ensuring that $\Sigma$ is regular are given in \cite{AD1, QS, AD2, Ex, ExNg, BrOz}. 
Inspired by these results, we will introduce a weakening of Exel's approximation property that is enough to ensure regularity.

\medskip We first record a trivial, but useful observation:

\begin{lemma} \label{weak} Let $\phi_1$ and $\phi_2$ be homomorphisms of $C^*(\Sigma)$. Then $\phi_1$ is weakly contained in $\phi_2$ if and only if there exists a net $\{\psi^i\}$ of  maps from $\phi_2(C^*(\Sigma))$ into $\phi_1(C^*(\Sigma))$ 
which satisfies
\begin{equation} \label{reg}
\lim_i \, \|\psi^i(\phi_2(x)) - \phi_1(x) \| = 0\, , \quad x \in C^*(\Sigma)\,.
\end{equation}

\end{lemma}

\begin{proposition} \label{netcond} 
Let $\phi$ be any faithful homomorphism  of $C^*(\Sigma)$.
The following conditions are equivalent:
\begin{itemize}
\item[(i)] $\Sigma$ is regular.
\item[(ii)] There exists  
a net $\{\psi^i\}$ of  maps from $C_r^*(\Sigma)$ 
into $\phi(C^*(\Sigma))$ 
such that
\begin{equation} \label{netreg}
\lim_i \, \|\psi^i(\Lambda(x)) - \phi(x) \| = 0\, , \quad x \in C^*(\Sigma)\,.
\end{equation}
\item[(iii)] There exists  
a net $\{\psi^i\}$ of  bounded linear maps from 
$ C^*_r(\Sigma)$   
into $\phi(C^*(\Sigma))$ 
which satisfy $\, \sup_i \|\psi^i\| < \infty$ and 
\begin{equation} \label{regu}
\lim_i \, \|\psi^i(\Lambda(f)) - \phi(f) \| = 0\, , \quad f \in C_c(\Sigma)\,.
\end{equation}
\end{itemize}
\end{proposition}

\begin{proof}
 If $\Sigma$ is regular, 
 then, by considering $\psi = \phi \circ \Lambda^{-1}$, we see that $(iii)$ holds. 
 Next, assume that $(iii)$ holds. 
 Using  that  $\, \sup_i \|\psi^i\| < \infty$, a standard $\varepsilon/3$-argument gives that $(\ref{netreg})$ follows from $(\ref{regu})$. Hence $(ii)$ holds.  
  Finally, if $(ii)$ holds, then Lemma \ref{weak} gives that $\phi$ is weakly contained in $\Lambda$, so $(i)$ holds. 
\end{proof}
\begin{remark} In  conditions $(ii)$ and $(iii)$ in Proposition \ref{netcond}, $\Lambda$ may be replaced  by any regular homomorphism of $C^*(\Sigma)$. 
\end{remark}

The following definition will be useful:
\begin{definition} Let  $T:G\times A \to A$ be a map 
which is linear in the second variable
and let $T_c: C_c(\Sigma) \to C_c(\Sigma)$ be the induced linear map defined by
\begin{equation}
[T_c(f)](g) = T(g, f(g)) \,, \quad f \in C_c(\Sigma)\, , \,g \in G \,.
\end{equation}
We will say that $T$ is a {\it rf-multiplier of $\Sigma$} whenever 
there exists a (necessarily unique) bounded linear map $\varphi_T: C_r^*(\Sigma) \to C^*(\Sigma)$ satisfying 
$$ \varphi_T(\Lambda(f))=T_c(f) \,, \quad f \in C_c(\Sigma)\,.$$
\end{definition} 
The existence of nonzero rf-multipliers of $\Sigma$ is not obvious, except when $G$ is amenable. Using our work in the previous section, we can show:

 \begin{proposition} \label{T-mult} Let $(\rho,v)$ be an equivariant representation of $\Sigma$ on a Hilbert $A$-module $X$ and let $\xi, \eta \in X^G$. 
Define  $T:G\times A \to A$  by 
$$T(g,a) = \langle \xi,\check\rho(a)\check{v}(g)\eta\rangle\, , \quad g \in G, \,  a \in A.$$
Then $T$ is a rf-multiplier of $\Sigma$ and $\|\varphi_T\| \leq \|\xi\| \, \|\eta\|$.
 \end{proposition}

\begin{proof} We first choose a faithful representation  $ \pi \times u$ of $C^*(\Sigma)$ on some Hilbert $C^*$-module $X$.  
According to Proposition \ref{Phi-prop}, there exists a (completely) bounded map $\Phi_T: \L(X^G)\to \L(Y)$  satisfying
\begin{equation} \label{Phi'}
\Phi_T( \big(\tilde\rho \times\tilde{\lambda}_\rho)(f)\big) = (\pi \times u) (T_c(f))\, , \quad f \in C_c(\Sigma)\,.
\end{equation}
Letting $\theta:C_r^*(\Sigma)\to (\tilde\rho \times\tilde{\lambda}_\rho)(C^*(\Sigma))$ denote the homomorphism satisfying $\theta\circ \Lambda = \tilde\rho \times\tilde{\lambda}_\rho$, it follows that the map $$\varphi_T= (\pi \times u)^{-1} \circ \Phi_T \circ \theta: C^*_r(\Sigma)\to C^*(\Sigma)\,$$ 
is  (completely) bounded, with $\|\varphi_T\| \leq \|\varphi_T\|_{\rm cb} =\| \Phi_T \|_{\rm cb}  \leq \|\xi\| \, \|\eta\|$, and satisfies
$$\varphi_T(\Lambda(f)) = T_c(f)\, , \quad f \in C_c(\Sigma).$$
The assertion is thereby proven.

\end{proof}

\begin{proposition}  \label{reg2} Let  $\{ T^i\}$ be a net of rf-multipliers of $\Sigma$ and set $\varphi^i= \varphi_{T^i}$ for each $i$.
Assume that   the following two conditions are satisfied:
\begin{itemize}
\item[(i)] $ \sup_i \| \varphi^i\|�  <  \infty $.
\item[(ii)] $\lim_i \| T^i(g,a) - a \|�= 0 \, , \quad g \in G\, , \,a \in A$.  
\end{itemize}
Then $\Sigma$ is regular. 
\end{proposition}
\begin{proof} 
Let $f \in C_c(\Sigma)$
and denote its support by $F$. Then $$\|\varphi^i(\Lambda(f))- f \|_* 
=\|\sum_{g \in F} \, \big(T^i(g,f(g)) -f(g)\big) \odot \delta_g \|_* 
\leq \sum_{g \in F} \, \|T^i(g,f(g)) -f(g)\| $$ 
Hence, it follows readily from $(ii)$ that  $\lim_i \| \varphi^i(\Lambda(f))- f \, \|_*�= 0$.
Taking into account $(i)$,
this means that condition $(iii)$ in Proposition \ref{netcond} is satisfied (with $\phi$ equal to the identity morphism). Hence $\Sigma$ is regular. 

\end{proof}
Conversely, if $\Sigma$ is regular, then a net satisfying all assumptions in  Proposition \ref{reg2}  trivially exists (as the map $I:G\times A \to A$ given by $I(g,a) = a$
 for all $g\in G, a\in A$, is  a rf-multiplier of $\Sigma$ in this case). 

\begin{definition}\label{weak-app}
 We will say that $\Sigma$ has the {\it  weak approximation property} if there exist an equivariant representation $(\rho, v)$ of $\Sigma$ on some Hilbert $A$-module $X$ and 
nets $\{\xi_i\}, \ \{\eta_i\} $ in $ X^G$ satisfying
\begin{itemize}
\item[a)] there exists some $ M > 0$ such that $\|\xi_i\| \cdot \|\eta_i\| \leq M $ \ for all $i $;
\item[b)] for all $g \in G$ and $a \in A$ we have $\lim_i \| \big\langle \xi_i\,,\,\check\rho(a)\check{v}(g)\eta_i \big\rangle  - a\|�= 0$, i.e.,
$$\, \lim_i   \sum_{h \in G} \big\langle \xi_i(h)\, , \,  \rho(a) \,v(g)\eta_i(g^{-1}h)\big\rangle \, =\,  a\, .$$

\end{itemize}
If one can choose $\eta_i = \xi_i$ for each $i$, we will say that $\Sigma$ has the {\it positive weak approximation property}.

\smallskip We will add the qualifying word {\it central} if the $\eta_i$'s and  the $\xi_i$'s can be chosen to lie in the central part of $X^G$.

\smallskip If $(\rho, v)$ can be chosen to be equal to $(\ell, \alpha)$, we will just talk about the corresponding {\it approximation property}.
\end{definition}  

\smallskip 
\begin{remark}
i) Trivially, $\langle 1\, , \ell(a)\, \alpha_g(1) \rangle = a $ 
for all $a \in A$ and $g \in G$. So the weak approximation property may be 
thought as expressing a kind of weak containment of the trivial equivariant representation  in some induced regular equivariant representation, 
hence as a form of amenability of $\Sigma$.

\smallskip ii) Without loss of generality, the nets $\{\xi_i\},\, \{\eta_i\} $ may both be assumed to lie in $C_c(G,X)$ (using that $C_c(G,X)$ is dense in $X^G$). 

\smallskip iii) Recall that $Z_X^{\, G}$ denotes the central part of $X^G$ (so  $\xi \in Z_X^{\,G}$ means that $\xi \in X^{G}$ and  $\rho(a) \,\xi(g) = \xi(g)\cdot a$ for all $\, a \in A$ and $\, g \in G$). 
Now, 
if all the $\xi_i$'s (or  the $\eta_i$'s) can be chosen 
to lie in $ Z_X^{\,G}$, 
then b) holds if and only if

\vspace{-1ex} \begin{equation}  
\lim_i \sum_{h \in G} \big\langle \xi_i(h), \, v(g) \eta_i(g^{-1}h) \big\rangle \, = \,1 
\end{equation}

\vspace{-1ex}for all $g \in G$.
\end{remark}

\begin{remark}\label{appdef}
The positive approximation property and the approximation property have previously been considered in the more general context of Fell bundles over discrete groups by Exel  \cite{Ex}�\, and over locally compact groups by Exel and Ng \cite{ExNg}.  In our setting, the approximation property of Exel says that
there exist
nets $\{\xi_i\},  \, \{\eta_i\}$ in $C_c(G,A)$ satisfying
\begin{itemize}
\item[i)] there exists some $ M > 0$ such that $\|\xi_i\| \cdot \|\eta_i\| \leq M$ \ for all $i$;
\item[ii)] $\, \lim_i \, \sum_{h \in G} \xi_i(gh)^* \, a \,\alpha_g(\eta_i(h))  = a\, $
for all $g \in G$ and $a \in A$.
\end{itemize}
This is easily seen to be equivalent to the definition of the approximation property we have given above. 
 Note that if  $\xi_i \in Z(A)^G$ (resp. $\eta_i \in Z(A)^G$) for each $i$, then ii) reduces to 
\begin{equation} \label{AD} 
  \lim_i \,  \sum_{h\in G}\, \xi_i(gh)^* \, \alpha_g(\eta_i(h))   = 1
\end{equation}
for all $g \in G$. It follows readily from this  that $\Sigma$ has the central positive approximation property whenever  $G$ is amenable. 
\end{remark}

\begin{remark} Assume that $\sigma$ is trivial.  In this case a strong form of the  central positive approximation property is discussed by Brown and Ozawa in their recent book \cite[Section 4.3]{BrOz}.
Their notion is closely related to the amenability of $\alpha$ as defined by  Anantharaman-Delaroche \cite{AD1}. 
When $A$ is abelian, these notions of amenability of $\alpha$ have been characterized in various ways: see e.g.  \cite{AD1, ExNg, BrOz, AD2}.
\end{remark}

The following result may be deduced from \cite{Ex} (see also \cite{ExNg}) in the case where $\Sigma$ has the approximation property. 

\begin{theorem} \label{mainreg} Assume that $\Sigma$ has the  weak approximation property. Then $\Sigma$ is regular. Moreover, $C^*(\Sigma) \simeq C^*_r(\Sigma)$ is nuclear if and only if $A$ is nuclear.
\end{theorem}

\begin{proof} The first assertion follows readily from Propositions \ref{T-mult}  and  \ref{reg2}. The second assertion may then be deduced from this in a standard way (see e.g.\ \cite{AD2, Ec, BrOz}).  
\end{proof}

As alluded to in the introduction, it is conceivable that the weak approximation property is equivalent to the approximation property. Even if this happens to be true, the weak approximation property should still be considered as a useful tool, as will be illustrated in Proposition \ref{perm} and its corollaries. It seems that it can be easier to check in certain cases, as will be illustrated in Example \ref{exact}.

The following proposition shows that the corresponding central properties are equivalent:

\begin{proposition}  \label{central} Let $\Sigma'= (Z(A), G, \alpha',1)$ be defined as in Section 4. The following conditons are equivalent:
\begin{itemize} 
\item[(a)] $\Sigma$ has the  central weak approximation property. 
\item[(b)] $\Sigma$ has the central approximation property. 
\item[(c)] $\Sigma'$ has the weak approximation property.
\item[(d)] $\Sigma'$ has the approximation property.
\item[(e)] $\alpha'$ is amenable in the sense of Anantharaman-Delaroche. 
\end{itemize}
\end{proposition} 
\begin{proof}
A moment's thought gives that $(b)$ is equivalent to $(d)$. The equivalence of $(d)$ and $(e)$ follows from \cite[Corollary 4.6]{ExNg}. The implication $(b) \Rightarrow (a)$ is trivial, while $(a) \Rightarrow (c)$ follows readily from Proposition \ref{centralmod}. So it suffices to show $(c) \Rightarrow (e)$. Assume that $\Sigma'$ has the weak approximation property. Then Theorem \ref{mainreg} gives that $\Sigma'$ is regular. As $Z(A)$ is nuclear, it follows from \cite[Theorem 4.5]{AD1} that $\alpha'$ is amenable in the sense of Anantharaman-Delaroche.
\end{proof}

The case where $\Sigma$ comes from a classical dynamical system, i.e.\ $A$ is abelian, 
has been studied by many authors, especially when $\sigma$ is trivial. Of course, the central approximation property and the approximation property are identical when $A$ is abelian. It may be worth stating explicitely  the following corollary (where the equivalence of $(a)$ and $(b)$ is due to Exel-Ng \cite{ExNg} in the untwisted case). 

\begin{corollary}  \label{central2} Assume $A$ is abelian. Then the following conditons are equivalent:
\begin{itemize} 
\item[(a)] $\Sigma$ has the  approximation property. 
\item[(b)] $\alpha$ is amenable in the sense of Anantharaman-Delaroche. 
\item[(c)] $\Sigma$ has the central weak approximation property. 

\end{itemize}
If $\sigma$ is scalar-valued, then any of these conditions is also equivalent to:

\begin{itemize} 

\item[(d)] $\Sigma$ has the weak approximation property.

\end{itemize}
\end{corollary} 
\begin{proof}
The equivalence of $(a), (b)$ and $(c)$ follows immediately from Proposition \ref{central}. The implication $(c) \Rightarrow (d)$ is trivial. If $\sigma$ is scalar-valued and $(d)$ holds, then $\Sigma' = (A, G, \alpha, 1)$ also has the weak approximation property, and  Proposition \ref{central} gives that $(b)$ holds.   

\end{proof}

By setting $A=\Complessi$ in Corollary \ref{central2}, we get that $(\Complessi, G, {\rm id}, \sigma)$ has the weak approximation property if and only if it has the approximation property, if and only if $G$ is amenable. Of course, this fact is just an easy consequence of the classical absorption principle. 

\smallskip In view of Proposition \ref{central}, we only refer to the central approximation property in our 
next result.

\begin{corollary} \label{amen} Assume that $\sigma$ is central (resp. $A$ has at least one tracial state). Then the following conditions are equivalent:\footnote{For this result to hold, it is important that  $A$ is unital, cf.\ \cite[Remark 5.3]{ZM}.} 
\begin{itemize}
\item[(a)]  $\Sigma$ has the central approximation property and there exists a state (resp.\ tracial state) on $A$ which is $\alpha$-invariant. 
\item[(b)]  $G$ is amenable. 
\end{itemize}
\end{corollary}
\begin{proof}  
Assume that $(a)$ holds. Then Proposition \ref{central} gives that $\Sigma'$ has the approximation property, hence that $\Sigma'$ is regular by Theorem \ref{mainreg}. Moreover, by restriction, there exists an $\alpha'$-invariant state on $Z(A)$. The amenability of $G$ follows then from \cite[Proposition 5.2]{ZM}. Hence $(b)$ holds.   

 Conversely, assume that $G$ is amenable. Then, as pointed out already, $\Sigma$ has the central approximation property. To show that there exists a state (resp.\ tracial state) on $A$ which is $\alpha$-invariant,
let $\varphi$ be a state (resp. tracial state) on $A$ and  $m$ a right translation invariant state on $\ell^{\infty}(G)$. 
For each $a\in A$ define $F(a) \in \ell^\infty(G)$ by $[(F(a)](g)= \varphi(\alpha_g(a)),\, g\in G$. Then, as is well known and easy to check, $\tilde{\varphi}= m \circ F$ gives a state (resp.�tracial state) on $A$. Moreover,
$$[F(\alpha_h(a))](g)=  \varphi(\alpha_g(\alpha_h(a))) = \varphi\big(\sigma(g,h)\alpha_{gh}(a))\sigma(g,h)^*\big) $$ $$= \varphi(\alpha_{gh}(a))= [F(a)](gh)$$ 
for each $a\in A, \, g,h \in G$. This means that $F(\alpha_h(a))$ is the right translate of $F(a)$ by $h$. Hence, the right invariance of $m$ gives that $\tilde{\varphi} $ is $\alpha$-invariant, as desired. 
\end{proof}

The assumptions in Corollary \ref{amen} are only used in the proof of $(b) \Rightarrow (a)$. 
One may wonder whether the following generalization of $(a) \Rightarrow (b)$ holds:  if $\Sigma$ has  the weak approximation property and there exists an $\alpha$-invariant state on $A$, then $G$ is amenable. We will show in Corollary \ref{inv} that this is true if we also assume that $\sigma$ is scalar-valued. 
We will first establish a permanence result for the weak approximation property, which illustrates the flexibility of this concept. 
Let us assume that the following conditions are satisfied: 
\begin{itemize}
\item $B$ is a C$^*$-subalgebra of $A$ containing the unit of $A$, 
\item  $B$ is invariant under each $\alpha_g, \, g\in G$. 
\item   $\sigma$  takes values in $\mathcal{U}(B)$.  

\end{itemize}
 Letting $\beta_g$ denote the restriction of $\alpha_g$ to $B$ for each  $g\in G$, we get a 
discrete twisted C$^*$-dynamical
 system $\Omega = (B, G, \beta, \sigma)$. We will call such a system for a $G$-{\it subsystem} of $\Sigma$ in the sequel.

\medskip A typical example of this situation is when  $B=Z(A)$ and $\sigma$ is assumed to be central. 
One may also consider 
the obvious product of two $G$-systems $(B, G, \beta, \sigma)$ and $(B', G, \beta', 1)$.

Note that $\Omega$  will not necessarily inherit the approximation property (or the weak approximation property) from $\Sigma$. For example, if $G$ is non-amenable, $B=\Complessi \cdot 1$ and $\sigma$ is scalar-valued, then $\Omega$ does not have the approximation property, while there are many known examples such that  $\Sigma$ do have it. 

To rule out this kind of example, we will require that there exists a conditional expectation $E:A \to B$  satisfying $E\circ \alpha_g = 
\beta_g \circ E$ for every $g\in G$ (i.e.\ $E$ is equivariant).  At first sight, it seems then reasonable that  if $\Sigma$ has the approximation property, then $\Omega$ will also have it, the reason being that if $\xi \in A^G$, then $\xi'= E\circ \xi$ is easily seen to lie in $B^G$. However, it is not obvious how to prove that  if $\{\xi_i\}, \, \{\eta_i\}$ are nets in $A^G$ witnessing the approximation property for $\Sigma$, then $\{\xi'_i\}, \{\eta'_i\}$ will be such nets for $\Omega$. 
Nevertheless, t 
we can show the following: 

\begin{proposition} \label{perm} Assume that
 $\Omega = (B, G, \beta, \sigma)$ is a  $G$-subsystem of $\Sigma$ with an equivariant conditional expectation $E:A \to B$. Then $\Omega$ has the weak approximation property whenever $\Sigma$ has the weak approximation property.
\end{proposition}

\begin{proof} Assume that  $\Sigma$ has the weak approximation property, so there exist an equivariant representation $(\rho, v)$ of $\Sigma$ on some Hilbert $A$-module $X$ and 
nets $\{\xi_i\}, \ \{\eta_i\} $ in $ X^G$ satisfying a) and b) in Definition \ref{weak-app} for some $M>0$. 

\medskip We first treat the case where $E$ is faithful. By localization (see \cite{La1}), we may then turn $X$ into a Hilbert $B$-module, with  inner product given by $$\langle x, x'\rangle_B = E\big(\langle x, x'\rangle\big), \quad x, x' \in X\,.$$
To avoid confusion, we will write $X_B$ to denote $X$ when it is considered as a Hilbert $B$-module, and set $\|x\|_B = \|\langle x, x\rangle_B\| ^{1/2} = \|E(\langle x, x'\rangle)\|^{1/2}$.
 
\medskip  Now, if  $T\in\L(X)$, then $$\langle Tx, x'\rangle_B = E\big(\langle Tx, x'\rangle\big)=  E\big(\langle x, T^*x'\rangle\big)= \langle x, T^*x'\rangle_B$$
for all $ x, x' \in X_B\,, $ so $T\in \L(X_B)$. 
Hence, we get a representation $\rho'$ from $B$ into $\L(X_B)$ by setting $\rho'(b)= \rho(b) \in \L(X_B)$ for each $b \in B$.  

\medskip Moreover, each $v(g)$ is an isometry as a map from $X_B$ into itself. To see this, note that
$$\langle v(g)x, v(g)x'\rangle_B = E\big(\langle v(g)x, v(g)x'\rangle\big) = E\big( \alpha_g(\langle x, x'\rangle) \big) $$ $$= 
\beta_g \big(E(\langle x, x'\rangle) \big) = \beta_g\big(\langle x, x'\rangle_B\big) $$
for all $g \in G, \, x, x' \in X_B\,.$ 
Thus, we get
$$\|v(g)x \|_B = \|\beta_g\big(\langle x, x\rangle_B\big)\|^{1/2} =\|\langle x, x\rangle_B\|^{1/2} = \|x\|_B $$
for all $g \in G, \, x \in X_B\,.$ 

\medskip Since each $v(g)$ is invertible as a map from $X_B$ into itself,   $v:g \to v(g)$ is a map from $G$ into $\mathcal{I}(X_B)$. It is then straightforward to check that $(\rho', v)$ is an equivariant representation of $\Omega$ on $X_B$. For instance, the computation we did previously shows that third condition holds.

\medskip We also note that if $\xi \in X^G$, then $\sum_{g\in G}\, \langle \xi(g), \xi(g)\rangle = a $ for some $a \in A$, and we get 
$$\sum_{g\in G}\, \langle \xi(g), \xi(g)\rangle_B=  \sum_{g\in G} \,E\big(\langle \xi(g), \xi(g)\rangle\big) = E\, \Big(\sum_{g\in G} \langle \xi(g), \xi(g)\rangle\Big)  = E(a) \in B$$
Hence, $\xi \in (X_B)^G$ and its norm $\|\xi\|_B$ in $(X_B)^G$ satisfies $$\|\xi'\|_B =\| \sum_{g\in G} \, \langle \xi(g), \xi(g)\rangle_B\|^{1/2} = \|E(a)\|^{1/2} \leq \|a\|^{1/2} = \|\xi\|$$
This means that  $\{\xi_i\}, \ \{\eta_i\} $ are nets in $ (X_B)^G$ satisfying 
$$\|\xi_i\|_B\cdot \|\eta_i\|_B \leq \|\xi_i\|\cdot \|\eta_i\| \leq M $$
for all $i$. Moreover, for any $b\in B$ and $g\in G$, we have
$$\,    \sum_{h \in H} \big\langle \xi_i(h)\, , \,  \rho'(b) \,v(g)\eta_i(g^{-1}h)\big\rangle_B \, =
\, E\Big(  \sum_{h \in G} \big\langle \xi_i(h)\, , \,  \rho (b) \,v(g)\eta_i(g^{-1}h)\big\rangle\Big)
$$
which converges to $E(b) = b$ in the norm of $B$. It follows that $\Omega$ has the weak approximation property, as desired.

\medskip Assume now that $E$ is not faithful. Then $\langle \cdot, \cdot\rangle_B$ gives only a semi inner product on $X$, but factoring out the kernel $N=\{ x \in X \mid \langle x, x\rangle_B=0\}$ and completing $X/N$ in the usual way (cf.\ \cite{La1}), we obtain a Hilbert $B$-module  $X'_B$ whose inner product satisfies
$$\big\langle x+N, x'+N\big\rangle'_B=\big \langle x, x'\big\rangle_B= E(\langle x, x'\rangle), \quad x,x' \in X.$$
 As shown in \cite[p. 57-58]{La1}, there exists a unital homomorphism $\pi_B: \L(X) \to \L(X'_B)$ satisfying  $[\pi_B(T)](x+N) =Tx + N$ for each $T\in \L(X), \, x \in X$. Hence, the map $\rho': B \to \L(X'_B)$ defined by $\rho'(b)= \pi_B(\rho(b))$ gives a representation of $B$ on $X'_B$. Moreover, it is clear that $v(g)N \subseteq N$ for each $g\in G$, so the map $v'(g) : x + N \to v(g)x + N$ is well defined on $X/N$. Since 
 $$\big\langle v'(g)(x+N), v'(g)(x'+N)\big\rangle'_B= \big\langle v(g)x, v(g)x'\big\rangle_B= \beta_g\big(\big\langle x, x'\big\rangle_B\big)$$ $$=\beta_g\big(\big\langle x+N, x'+N\big\rangle'_B\big) $$ for all $ x,x' \in X$, 
 it follows readily that $v'(g)$ extends to an isometry on $X'_B$ for each $g\in G$. 
 
 A straightforward computation gives that $v'(g)v'(h)= {\rm ad}_{\rho'}(\sigma(g,h)) v'(gh)$ on $X/N$, and therefore also on $X'_B$ by density. We especially have $v'(e)v'(h)=v'(h)$ for all $h\in G$, and it follows easily from this and the invertibility of $v(h)$ on $X$ that $v'(e)$ is the identity operator on $X'_B$. So for each $g\in G$ we have $v'(g) v'(g^{-1}) = {\rm ad}_{\rho'}(\sigma(g,g^{-1}))$, which gives that $v'(g):X'_B\to X'_B$ is invertible with $v'(g)^{-1}= v'(g^{-1})\, {\rm ad}_{\rho'}(\sigma(g,g^{-1})^*)$.  Thus $v':g\to v'(g)$ is a map from $G$ into $\mathcal{I}(X'_B)$. 
 
 \medskip It is now a routine exercise to proceed further, essentially as in the faithful case, and reach the desired conclusion.

\end{proof}

\begin{corollary} \label{abel} Assume that $\sigma$ is scalar-valued. 
  Consider a  $G$-subsystem $\Omega = (B, G, \beta, \sigma)$ of $\Sigma$  with $B$ abelian, and assume that there exists an equivariant conditional expectation $E:A \to B$. 
 
  Then $\Omega$ has the approximation property (equivalently, $\beta$ is amenable in the sense of Anantharaman-Delaroche) whenever $\Sigma$ has the weak approximation property (hence, especially when $\Sigma$ has the approximation property).
\end{corollary}
\begin{proof}
This follows immediately from Corollary \ref{central2} and Proposition  \ref{perm}.

\end{proof}  

\begin{corollary}  Assume that $\sigma$ is scalar-valued and that there exists an equivariant conditional expectation $E:A\to Z(A)$. 
Then all the approximation properties for $\Sigma$  (central or not) are equivalent, being all equivalent to the amenability of $\alpha'$ in the sense of Anantharaman-Delaroche.

\end{corollary}
\begin{proof}
By considering the $G$-subsystem $(Z(A), G, \alpha', \sigma)$, the assertion follows  from Corollary \ref{abel} and Proposition \ref{central}.

\end{proof}

\medskip We obtain the result mentioned after Corollary \ref{amen}:

\begin{corollary} \label{inv} Assume that $\sigma$ is scalar-valued. If $\Sigma$ has  the weak approximation property and there exists an $\alpha$-invariant state on $A$, then $G$ is amenable.
\end{corollary}
\begin{proof}
It suffices to use Corollary \ref{abel} on the $G$-subsystem $(\Complessi\cdot1, G, {\rm id}, \sigma)$.

\end{proof}

We conclude the paper with two examples. 

\begin{example} \label{exact}

Let $G$ be an exact group, $\sigma \in Z^2(G, \Toro)$ and  $\alpha$ denote the action of $G$ on $A=\ell^\infty(G)$ by left translations. Then $\alpha$ is amenable in the sense of Anantharaman-Delaroche (see \cite{BrOz} and  \cite{AD2}, and references therein). Hence  
$\Sigma=(A, G, \alpha, \sigma)$ has the approximation property, as follows from Corollary \ref{central2}.

Now, let $H$ be any amenable subgroup of $G$ and let $\gamma$ denote the action of $H$ on $A$ by right translations.
Since the actions $\alpha$ and $\gamma$ commute, the fixed-point algebra $B=A^\gamma$ of $\gamma$ is $\alpha$-invariant. Hence we get an action $\beta$ of $G$ on $B$ by restricting $\alpha$ to $B$, and  $(B, \beta, G, \sigma)$ is a $G$-subsystem of $\Sigma$.  Note that $B$ consists of those functions in $\ell^\infty(G)$ that are constant on left $H$-cosets of $G$, so $B$ may be identified with $\ell^\infty(G/H)$, and $\beta$ with the natural action of $G$ on $\ell^\infty(G/H)$.

Now, as $A=\ell^\infty(G)$ is a W$^*$-algebra 
and $H$ is amenable, there exists a conditional expectation $E$ from $A$ onto $B$ (see e.g.\ \cite{S}). 
Such an $E$ may be constructed as follows. Let $m$ be any left-invariant mean on $\ell^\infty(H)$. For each $f\in A=\ell^\infty(G)$ and $g \in G$, let $f_{H,g} \in \ell^\infty(H)$ be defined by $f_{H,g}(h)= f(gh),\, h \in H$. Then $E$ may be defined by
$$[E(f)](g)= m\big(f_{H,g})\,, \quad f\in A\, , \, g\in G\,.$$
Let us check that $E$ is equivariant. Let $r, g \in G, f\in A$. Then, as
$$[\alpha_r(f)]_{H,g}(h)= [\alpha_r(f)](gh)= f(r^{-1}gh)= f_{H,r^{-1}g}(h)$$
for each $h\in H$, we have $[\alpha_r(f)]_{H,g} = f_{H,r^{-1}g}$. Thus we get
$$E(\alpha_r(f))(g) = m\big( [\alpha_r(f)]_{H,g}\big) = m\big(f_{H,r^{-1}g}\big)  
= [E(f)](r^{-1}g)= [\beta_r\big(E(f)\big)](g)$$
This shows that $E\circ\alpha_r=\beta_r\circ E$ for all $r \in G$. 

From Corollary \ref{abel} we can now conclude that  $\beta$
is amenable in the sense of Anantharaman-Delaroche. 
 This fact may be known to specialists, but it is not clear to us how to deduce it in an easier way (unless of course if $G$ is amenable). 
Moreover, it follows from Corollary \ref{abel} that $\Omega= (B, G, \beta, \sigma)$ has the approximation property. As $B$ is nuclear, using Theorem \ref{mainreg}, we get  that $C^*(\Omega) \simeq C_r^*(\Omega)$ is nuclear. Thinking of $B$ as $\ell^\infty(G/H)$ and $\beta$ as the natural action of $G$ on it, this result seems to be new, except in the "classical" case where 
$H$ is trivial and $\sigma=1$ (see e.g.\ \cite{BrOz, AD2}), and in the case where $H$ is trivial but $\sigma$ is not (see \cite{KLO}). 
\end{example}

\begin{example}
Proposition \ref{perm}� may be applied in the following situation.
Consider a system $\Sigma=(A, G, \alpha, \sigma)$ where  $\sigma$ is scalar-valued and assume that there exists a continuous action $\gamma$ of a compact group $K$ on $A$ which commutes with $\alpha$. Let $B=A^\gamma$ denote the fixed-point algebra of $\gamma$. Then $B$ is $\alpha$-invariant, so we get an action $\beta$ of $G$ on $B$ by restriction, i.e.\  $(B, \beta, G, \sigma)$ is $G$-subsystem of $\Sigma$.  
Now, by compactness of $K$, there exists a  (faithful) canonical conditional expectation $E$ from $A$ onto $B=A^\gamma$ given by $ E(a) = \int_K\gamma_g(x) \, dg, \, a \in A$. It is quite easy to check that $E\circ \alpha_r = \beta_r\circ E$ for all $r \in G$. Hence, we can conclude that $(B, G, \beta, \sigma)$ has the weak approximation property whenever $\Sigma$ has it. 
 
 As a concrete case, let $A=\mathcal{O}_n$ be the Cuntz algebra with $n$ generators for some $2\leq n < \infty.$
 Let $U(n)$ denote the group of all $n\times n$ unitary matrices. To any $U \in U(n)$ we may associate an automorphism $\alpha_U$ of $A$, often called a quasi-free automorphism (see e.g. \cite{ETW, Ev}),� such that $U\to \alpha_U$ is an (outer) action of $U(n)$ on $A$.  
 Let $\gamma$ denote the canonical gauge action of $\Toro$ on $A$ (i.e. $\gamma_z = \alpha_{zI_n}, \, z \in \Toro$). As is well-known, the fixed-point algebra $B=\mathcal{F}_n$ of   $\gamma$ is a UHF-algebra (of type $n^\infty$). 
 
 Now, let $G$ be any subgroup of $U(n)$ and $\alpha$ be the action of $G$ (as a discrete group)  on $A$ by quasi-free automorphisms. Since $\alpha$ and $\gamma$ commute, we are in the above situation (with $\sigma =1$), and we then know that  $(B, G, \beta, 1)$  will have the weak approximation property if $(A, G, \alpha, 1)$ has it. However, being UHF, $B$ has a unique tracial state, which is necessarily  $\beta$-invariant. Hence, Corollary \ref{inv} gives that $G$ must be amenable if $(B, G, \beta, 1)$ has the weak approximation property. Altogether, this means that the following holds:  $(A, G, \alpha, 1)$ has the weak approximation property if and only if 
 $G$ is amenable.  Especially, if $G$ is any non-amenable subgroup of $U(n)$, then $(A, G, \alpha, 1)$ does not have the weak approximation property.
  \end{example}

\bigskip

\bigskip

\noindent{\bf Acknowledgements.} 
 Most of the present work has been done during the several visits
 made by R.C. in 2009 and 2010 at the  Institute of Mathematics, University of Oslo.
 He thanks the operator algebra group for their kind hospitality and the Norwegian Research Council for partial financial support.

\bigskip
{\parindent=0pt Addresses of the authors:\\

\smallskip Erik B\'edos, Institute of Mathematics, University of
Oslo, \\
P.B. 1053 Blindern, N-0316 Oslo, Norway.\\ E-mail: bedos@math.uio.no. \\

\smallskip \noindent
Roberto Conti, Dipartimento di Scienze di Base e Applicate per l'Ingegneria, \\
Sezione di Matematica, 
Sapienza Universit\`a di Roma \\
Via A. Scarpa 16,
I-00161 Roma, Italy.
\\ E-mail: roberto.conti@sbai.uniroma1.it
\par}

\end{document}